\documentclass[11pt,twoside,a4wide]{amsart}
\usepackage{amssymb}
\usepackage[pdftex]{hyperref}
\newtheorem{thm}{Theorem}[section]
\newtheorem{prop}[thm]{Proposition}
\newtheorem{lem}[thm]{Lemma}

\theoremstyle{definition}

\newtheorem{ex}[thm]{Example}
\newtheorem{rem}[thm]{Remark}

\title{Some algebraic consequences \\of Green's hyperplane restriction theorems}

\author{Mats Boij}
\address{Department of Mathematics\\
Royal Institute of Technology\\
S-100 44 Stockholm, Sweden\\
E-mail: boij@kth.se\\
}
\author{Fabrizio Zanello}
\address{Department of Mathematical Sciences\\
Michigan Technological University\\
Houghton, MI 49931-1295\\
E-mail: zanello@mtu.edu}

\begin{document}
\begin{abstract}
We discuss M. Green's paper \cite{Gr} from a new algebraic perspective, and provide applications of its results to level and Gorenstein algebras, concerning their Hilbert functions and the weak Lefschetz property. In particular, we will determine a new infinite class of symmetric $h$-vectors that cannot be Gorenstein $h$-vectors, which was left open in the recent work \cite{MNZ1}. This includes the smallest example previously unknown, $h=(1,10,9,10,1)$. As M. Green's results depend heavily on the characteristic of the base field, so will ours. The appendix will contain a new argument, kindly provided to us by M. Green, for Theorems 3 and 4 of \cite{Gr}, since we had found a gap in the original proof of those results during the preparation of this manuscript. 
\end{abstract}

\maketitle

\section{Introduction}
The goal of this paper is to explore some consequences of the beautiful results of Mark Green contained in the 1988 paper \cite{Gr}, which supply upper bounds on the dimensions of linear series on projective spaces obtained by hyperplane restrictions. We will rephrase those results in an algebraic language that will allow us to obtain some interesting applications to the study of graded level, and especially Gorenstein, algebras.

The results contained in \cite{Gr}, especially the last two  - even though they are very natural and in the same spirit of the very classical theorems of Macaulay and Gotzmann - have not been very much employed in commutative algebra after their publication, except perhaps only by Bigatti-Geramita-Migliore~\cite{BGM}. 

Theorem 1 of [8] has been widely used in the literature. For instance, it has been a key tool to give simpler proofs of Gotzmann's persistence theorem (a result originally proved in \cite{Go} over an arbitrary Noetherian ring) when the base ring is a field (see \cite{BH,
IK,IKl}). Also, see the subsequent generalizations of Green's theorem by Herzog and Popescu~\cite{HP} in characteristic zero and Gasharov~\cite{Ga} in characteristic $p$. More recently, the second author, along with Migliore and Nagel, has extensively used Theorem 1 of \cite{Gr} in a few works concerning the study of artinian algebras (see \cite{MZ,MNZ1,MNZ2}).

In this paper we will discuss the two other main results of Green's paper (\cite[Theorems 3 and 4]{Gr}), by presenting them in a current form more useful to our purposes, and we will rely on them to show some new properties of level and Gorenstein algebras. In the preparation of this manuscript, we found gaps in the original proofs of those two theorems of Green's  that we were unable to fix. However, M. Green has recently provided us (personal communication) with a revised argument that shows that Theorem 3 is true in any characteristic different from two, and that Theorem 4 holds in characteristic zero and large enough positive characteristic. As a consequence of this, our results will have the same dependence of the characteristic. We will include a sketch of the new proofs for those two theorems of Green in an appendix to this paper.

Green's results will prove themselves very useful here in the study of some limit cases for Hilbert functions and even for the weak Lefschetz property. The most important of our results is perhaps that of an infinite class of symmetric $h$-vectors which cannot be the $h$-vector of a Gorenstein algebra, filling a gap left open in  the latest work in this direction (see \cite[Question 12 and the comment following it]{MNZ1}.) In particular, the smallest candidate whose nature was not determined in \cite{MNZ1}, $h=(1,10,9,10,1)$, is {\em not} a Gorenstein $h$-vector.

We believe that Green's paper will deserve further attention in the future, possibly from a still different perspective, since we are sure that we have not fully exploited the power of Green's results in this article.

\section{M. Green's results and first applications}
Before discussing Green's results, we now need to recall the definitions and the main facts we will use throughout this paper. We consider standard graded artinian algebras $A=\bigoplus_{i=0}^e A_i=R/I$, where $R=k[x_1,\dots,x_r]$, $k$ is (unless otherwise specified) an infinite field, $I$ is a homogeneous ideal of $R$, and the $x_i$'s have degree 1. 

We write the {\em Hilbert function} (or {\em $h$-vector}, since we are in the artinian case) of $A$ as $h(A)=h=(h_0,h_1,\dots,h_e)$, where $h_i=\dim_k A_i$ and $e$ is the last index such that $\dim_k A_e>0$. Since we may suppose that $I$ does not contain non-zero forms of degree 1, $r=h_1$ is defined as the {\em codimension} of $A$.

The {\em socle} of $A$ is the annihilator of the maximal homogeneous ideal $\overline{m}=(\overline{x_1},\dots,\overline{x_r})\subseteq A$, namely $soc(A)=\lbrace a\in A {\ } \mid {\ } a\overline{m}=0\rbrace $. Since $soc(A)$ is a homogeneous ideal, we can define the {\em socle-vector} of $A$ as $s(A)=s=(s_0,s_1,\dots,s_e)$, where $s_i=\dim_k soc(A)_i$. Notice that $h_0=1$, $s_0=0$ and $s_e=h_e>0$. The integer $e$ is called the {\em socle degree} of $A$ (or of $h$). 
 
If $s=(0,0,\dots,0,s_e=t)$, we say that the algebra $A$ (or its $h$-vector) is {\em level}. In particular, the case $t=1$ is called {\em Gorenstein}.

An artinian algebra $A=\bigoplus_{i=0}^e A_i$ is said to have the {\em weak Lefschetz property} ({\em WLP}) if there exists a linear form $L\in R$ such that, for all indices $i=0,1,\dots,e-1$, the multiplication map \lq \lq $\cdot L$" between the $k$-vector spaces $A_i$ and $A_{i+1}$ has maximal rank. In this case, the general linear form has this property.

Lots of research has been performed over the last few years in order to understand the structure of Gorenstein algebras, their $h$-vectors and the existence of the weak Lefschetz property. The main contributions to the subject, begun with Stanley's seminal paper \cite{St1}, include \cite{BI,Bo,BL,BE,Ha,HMNW,IK,IS,MN,MNZ1,MNZ2,Za}.

Recall that, for $n$ and $i$ positive integers, the {\em i-binomial expansion of n} is 
\[
n_{(i)} = \binom{n_i}{ i}+\binom{n_{i-1}}{i-1}+\dots+\binom{n_j}{j},
\]
 where $n_i>n_{i-1}>\dots>n_j \geq j \geq 1$. It is a standard fact that such an expansion always exists and is unique (see, e.g.,  \cite[Lemma 4.2.6]{BH}).
 
Following \cite{BG}, define, for any integers $a$ and $b$,
\[
\left(n_{(i)}\right)_{a}^{b}=\binom{n_i+b}{i+a}+\binom{n_{i-1}+b}{i-1+a}+\dots+\binom{n_j+b}{j+a},
\]
where, as usual, we set $\binom{m}{q}=0$ whenever $m<q$ or $q<0$.

Let us finally recall Macaulay's theorem, which characterizes the sequences of integers that may occur as Hilbert functions of  standard graded artinian algebras (Macaulay's result actually holds, with the obvious  modifications, also in the non-artinian case):

\begin{thm}[Macaulay]\label{mac}
Let $h=(h_i)_{0\leq 1\leq e}$ be a sequence of positive integers, such that $h_0=1$ and $h_1=r$. Then $h$ is the Hilbert function of some standard graded artinian algebra if and only if, for every $d\geq 1$,
$$h_{d+1}\leq \left((h_d)_{(d)}\right)^{1}_{1}.$$
\end{thm}

\begin{proof}
See \cite[Theorem 4.2.10]{BH}.
\end{proof}

Let us now begin by rephrasing Green's first hyperplane  restriction theorem~\cite[Theorem 1]{Gr}
in the language of this paper. It supplies an upper bound for the Hilbert function of the quotient of a given graded algebra (again, not necessarily artinian) by a general linear form. 

\begin{thm}\cite[Theorem 1]{Gr}\label{1}
Let $h_d$ be the entry of degree $d$ of the Hilbert function of $R/I$, and let $L$ be a general linear form of $R$. Then the degree $d$ entry, $h_d^{'}$, of the Hilbert function of $R/(I,L)$ satisfies the inequality
$$h_d^{'}\leq \left((h_d)_{(d)}\right)^{-1}_{0}.$$
\end{thm}

As Green remarked in his paper, the bound provided by the above theorem can be sharp, even simultaneously, in all degrees (indeed, we know that for instance it is achieved when the ideal $I$ is a lex-segment). The next two results of Green, that we are going to present below, provide very useful information on the ideal $I$ in two particular cases when the upper bound is sharp in some specific entry.
With our notation, we have:

\begin{thm}\cite[Theorem 3]{Gr}\label{3} Assume that $\operatorname{char} k\ne 2$. 
If, for some positive integer $m$, $h_d=\binom{m+d}{d}$, and $h^{'}_d=\binom{m-1+d}{d}$, then the degree $d$ graded piece of the ideal $I$ is spanned by an $m$-dimensional linear space. In other words,
$$I_d=R_{d-1}\cdot \langle L_1,\dots,L_{r-m-1}\rangle,$$
where $L_1,L_2,\dots,L_{r-m-1}$ are linearly independent linear forms of $R$.
\end{thm}

\begin{rem} \label{rem2.4}
  Notice that the conclusion of Theorem~\ref{3} no longer holds if the base field $k$ has characteristic two, since in this case the space $V=(x_1^2,x_1^2,\dots,x_r^2)$ satisfies the conditions of the theorem while it is not generated by linear forms. 
\end{rem}

As far as (artinian) level algebras are concerned, this result of Green's has the following important application concerning their Hilbert functions:

\begin{prop}\label{pr3}
With the notation and under the hypotheses of Theorem \ref{3}, suppose that the algebra $A=R/I$ having $h$-vector $h$ has only a zero socle in all degrees $\leq d-1$ (for instance, when $A$ is level). Then all the entries of $h$ up to degree $d$ are determined. Namely, we have: $$h=\left(1,h_1=m+1,h_2=\binom{m+2}{2},\dots,h_{i}=\binom{m+i}{i},\dots,h_d=\binom{m+d}{d}\right).$$
\end{prop}

\begin{proof} 
We know by Theorem~\ref{3} that $I_d = \langle L_1,\dots,L_{r-m-1}\rangle R_{d-1}$. Since there is a zero socle in all degrees less than $d$, we conclude that $I\subseteq(L_1,\dots,L_{r-m-1})$. Therefore, $I_i\subseteq(L_1,\dots,L_{r-m-1})_i$ for all $i\leq d$. On the other hand, there can be no other generators of $I$ in degrees between $1$ and $d$, since we have equality in degree $d$. Hence $I_i=(L_1,\dots,L_{r-m-1})_i$ for $0\le i\le d$, and the Hilbert function $h$ must be the Hilbert function of a linear space of dimension $m$, i.e., $h_i=\binom{m+i}{i}$. 
\end{proof}

Let us now present the last Theorem of \cite{Gr}, which describes the degree $d$ graded piece of $I$ when, with  the usual notation, $h_d$ and $h^{'}_d$ are of another particular form ($h^{'}_d$ again being the maximum allowed by Theorem \ref{1}). Again, Green's result will be presented in a different way from the original paper, consistently with our notation. Notice that the integer $m$ below will not run from $m=0$ (as stated in \cite{Gr}, probably because of a typo), but from $m=1$.

\begin{thm}\cite[Theorem 4]{Gr}\label{4} Assume that $\operatorname{char} k=0$\footnote{In fact, it is sufficient to assume $\operatorname{char} k\ge m$. Notice that the example of Remark \ref{rem2.4} can easily be generalized to arbitrary characteristic to show that the assumption on the characteristic that we make for this theorem is again necessary.} and suppose that, for some integer $m$, $1\leq m \leq d$, we have
$$h_d=md+1-\binom{m-1}{2}=
\binom{d+1}{d}+\binom{d}{d-1}+\dots+\binom{d-(m-2)}{d-(m-1)}$$
and $h^{'}_d=m$ (the latter being the maximal possible value according to Theorem \ref{1}). 
Then, in degree $d$, $I$ is the ideal of a plane curve of degree $m$. In other words,
$$I_d=\langle L_1,L_2,\dots,L_{r-3}\rangle R_{d-1} +F\cdot R_{d-m},$$
where $L_1,L_2,\dots,L_{r-3}$ are linearly independent linear forms and $F$ is a form of degree $m$.
\end{thm}

This result of Green's also has a very interesting general consequence for level algebras:

\begin{prop}\label{pr4}
Assume that $\operatorname{char}k=0$ and suppose that the algebra $A=R/I$ having $h$-vector $h$ has a zero socle in all degrees $\leq d-1$. If $h_d=\binom{d+2}{2}-\binom{d-m+2}{2}$ for some positive integer $m\le d$, and $\dim_k (A/(L))_d=m$ for a general linear form $L$, then 
$$
h_i = \binom{i+2}{2}-\binom{i-m+2}{2}, \qquad 0\le i \le d,
$$
i.e., $h$ equals the Hilbert function of a plane curve of degree $m$ up to degree $d$.
\end{prop}

\begin{proof}
It follows from Theorem~\ref{4} that $I_d=(L_1,\dots,L_{r-3},F)_d$, but because the socle is trivial in all degrees less than $d$, we have that $I\subseteq (L_1,\dots,L_{r-3},F)$. Since we have equality in degree $d$, there can be no more generators in degrees less than $d$, and therefore $I_i=(L_1,\dots,L_{r-3},F)_i$ for $i=0,1,\dots,d$. Thus, the Hilbert function equals the Hilbert function of a plane curve of degree $m$, as desired. 
\end{proof}

\begin{rem}
Notice that when $m=1$, Theorems \ref{3} and \ref{4} coincide. Therefore, the $h$-vectors provided by Propositions \ref{pr3} and \ref{pr4}, as it can be easily checked, also coincide in that special case.
\end{rem}

In order to use the previous results in the next section, we still need the following lemma, due to R. Stanley (see \cite[bottom of p. 67]{St2}):

\begin{lem}[Stanley]\label{sta}
Let $A=R/I$ be an artinian Gorenstein algebra, and let $L\notin I$ be a linear form of $R$. Then the  $h$-vector of $A$ can be written as
$$h := (h_0,h_1,\ldots,h_e) = (1, b_1+c_1,\dots, b_e+c_e = 1),$$ where
$$b=(b_1=1,b_2,\dots,b_{e-1},b_e=1)$$
is the $h$-vector of $R/(I:L)$ (with the indices shifted by 1), which is a Gorenstein algebra, and
$$c=(c_0=1,c_1,\dots,c_{e-1},c_e= 0)$$
is the $h$-vector of $R/(I,L)$.
\end{lem}

\begin{proof}
The fact that $R/(I:L)$ is Gorenstein of socle degree $e-1$ is well-known and easy to prove (in fact, an analogous conclusion holds, more generally, for any form $F\notin I$, not necessarily of degree 1). The decomposition of $h$ is an easy consequence of the exact sequence
$$0\longrightarrow R/(I:L) (-1)\stackrel{\cdot L}{\longrightarrow} R/I \stackrel{}{\longrightarrow} R/(I, L) \longrightarrow 0,$$
which is induced by  multiplication by $L$.
\end{proof}

\section{The main results}

The goal of this section is to deduce some deeper limit consequences from the results of Green that we presented in the previous section. From now on we will specifically focus on Gorenstein and level algebras.

In \cite[Theorem 4]{MNZ1}, Migliore, Nagel and the second author proved a strong lower bound for the second entry, $h_2$, of any Gorenstein $h$-vector of given codimension $h_1=r$ and socle degree $e$. Namely, they showed that
$$h_2\geq \left(r_{(e-1)}\right)^{-1}_{-1}+\left(r_{(e-1)}\right)^{-(e-2)}_{-(e-3)}.$$
Then, in Question 12 of the same article they asked whether their bound is actually always sharp.

In this paper, thanks to the machinery developed in the previous section, we are able to answer negatively to the above question in infinitely many cases, by supplying a new class of $h$-vectors which cannot be Gorenstein. We have:

\begin{thm}\label{10}  
Assume that $\operatorname{char}k\ne 2$. Then, for all integers $m\geq 2$,
$$\left(1,\binom{m+3}{3},(m+1)^2, \binom{m+3}{3},1\right)$$
is not a Gorenstein $h$-vector.
\end{thm}

\begin{rem}\label{1010} It is easy to compute that
$$\left(\binom{m+3}{3}_{(3)}\right)^{-1}_{-1}+\left(\binom{m+3}{3}_{(3)}\right)^{-2}_{-1}=(m+1)^2.$$
Hence, as we said above, Theorem \ref{10} answers \cite[Question 12]{MNZ1} when the codimension is $r=\binom{m+3}{3}$, for any $m\geq 2$, and the socle degree is $4$, in the sense that the bound given by \cite[Theorem 4]{MNZ1} is not sharp in these cases.

In particular, when $m=2$, our result also proves that the smallest example  from \cite{MNZ1}, in terms of codimension and socle degree, of an $h$-vector whose Gorensteinness was still undecided, i.e. $(1,10,9,10,1)$, is {\em not} Gorenstein. Now, the smallest unknown case is $(1,11,10,11,1)$.
\end{rem}

\begin{proof}[Proof of Theorem \ref{10}] Suppose the vector of the statement is the $h$-vector of a Gorenstein algebra $R/I$. We want to reach a contradiction. 
By Lemma \ref{sta}, we can decompose 
$$h=\left(1,\binom{m+3}{3},(m+1)^2, \binom{m+3}{3},1\right)$$ as $b+c=(1,b_2,b_3,1)+(1,c_1,c_2,c_3,c_4=0).$
Since $b$ is the $h$-vector of $R/(I:L)$, which is Gorenstein, by symmetry we have $b_2=b_3$.

Let us choose the form $L$ to be general in $R$. Hence, by Green's Theorem \ref{1}, we have that
$$c_3\leq \left( \binom{m+3}{3}_{(3)}\right)_{0}^{-1}=\binom{m+2}{3}.$$

Since the truncation of a Gorenstein algebra is a level algebra (this fact is a standard one, and can be immediately seen, for instance, using Macaulay's inverse systems), we can apply Proposition \ref{pr3}, with $d=3$. Therefore, since $\binom{m+3}{3}>m+1$ for all positive integers $m$, we would obtain a contradiction from the equality $c_3=\binom{m+2}{3}.$ Thus, $c_3<\binom{m+2}{3}.$

Since $b_2=b_3$, we have that $$c_3-c_2=\binom{m+3}{3}-(m+1)^2=\binom{m+2}{3}-\binom{m+1}{2}.$$
Hence, if we set $c_3=\binom{m+2}{3}-\alpha $, for some integer $\alpha >0$,  we have $c_2=\binom{m+1}{2}-\alpha $.

A standard computation shows that 
$$\left(\left(\binom{m+1}{2}-1   \right)_{(2)} \right)^1_1=\binom{m+2}{3}-m.$$
Since, for any $d$, $(a_{(d)})^1_1$ is a strictly increasing function of $a$, it follows that
$$((c_2)_{(2)})^1_1=\left(\left(\binom{m+1}{2}-\alpha   \right)_{(2)} \right)^1_1<\binom{m+2}{3}-\alpha =c_3,$$
a contradiction to Macaulay's Theorem \ref{mac}.\end{proof}

The next result is another interesting (and fairly general) consequence of Proposition \ref{pr3}, which also concerns the weak Lefschetz property.

\begin{thm}\label{wlp} 
Assume that $\operatorname{char}k \ne 2$. If $h=(1,r,h_2,\dots,h_e)$ is a Gorenstein $h$-vector such that $r=e\geq 2$, then $h_2\geq e$. Moreover, if $h_2=e$, then:
\begin{itemize}
\item[$i)$] $h_i=e$ for all indices $i=1,2,\dots,e-1$.
\item[$ii)$] Any Gorenstein algebra $A$ with $h$-vector $h$ has the weak Lefschetz property.
\end{itemize}
\end{thm}

\begin{proof} Notice that $h_2\geq e$ immediately follows from the inequality of \cite[Theorem 4]{MNZ1} (which we stated above). We will prove it again here though, since our argument will be useful to prove the rest of the statement. Again, with the usual notation, decompose $h$ as $b+c$ as in Stanley's Lemma \ref{sta}, where $L$ is a general linear form of $R$. Hence, by Green's Theorem \ref{1}, we have $c_{e-1}\leq \left(e_{(e-1)}\right)_0^{-1}=\binom{e-1}{e-1}=1$.

Thus, we can apply Proposition \ref{pr3}, which easily implies that $c_{e-1}=0$. 

It follows, by symmetry, that
$$h_2\geq b_2=b_{e-1}=h_{e-1}-c_{e-1}=h_{e-1}=h_1=e,$$ as desired.

Let us now suppose that $h_2=e$, and let us prove $i)$ and $ii)$. Since, as we just saw, $b_2=e$, we have $c_2=h_2-b_2=e-e=0$. Hence, by Macaulay's theorem, $c_i=0$ for all $i\geq 2$.

Thus, for every $i=2,\dots,e-1$, $$h_i=b_i-c_i=b_i=b_{e+1-i}=h_{e+1-i}=h_{i-1}.$$
Since $h_2=e$, by induction we immediately obtain $h_i=e$ for all $i=1,2,\dots,e-1$, which shows $i)$.

The existence of the weak Lefschetz property for any Gorenstein algebra with this $h$-vector $h$ is now immediate, since we have shown that $c=(1,e-1,0,0,\dots,0)$. This proves $ii)$ and the theorem.\end{proof}

\begin{rem} For $r=3$, the above theorem implies that all Gorenstein algebras with Hilbert function $(1,3,3,1)$ have the weak Lefschetz property, as also shown in \cite{MZ2}. This result, however, is false in characteristic two, since $R[x,y,z]/(x^2,y^2,z^2)$ does not have the weak Lefschetz property. Notice that this  is also a counterexample in characteristic two to both Theorems 3 and 4 of \cite{Gr}. 
\end{rem}

\begin{prop}\label{pr7} Assume that $\operatorname{char} k\ne 2$. 
Let $A$ be a level algebra of codimension at least three with Hilbert function $(h_0,h_1,\dots,h_d)$, where $h_{d-1}\leq h_d=d+1$. Then $h_0\leq h_1\leq \dots \leq h_{d-1}=h_d$ and $A$ has the weak Lefschetz property.
\end{prop}

\begin{proof}
  Let $L$ be a general linear form. Then $\dim_k (A/(L))_d\le 1$ according to Green's Theorem~\ref{1}. By Proposition~\ref{pr3}, we have that $\dim_k (A/(L))_d=1$ implies that $h_i=i+1$ for all $i$, contradicting the assumption that the codimension is at least three. Hence the multiplication by $L$ from $A_{d-1}$ to $A_d$ is surjective, but since $h_{d-1}\le h_d$, we have to have that $h_{d-1}=h_d$. Because the multiplication by $L$ is now also injective in degree $d-1$, it has to be injective in all lower degrees by the assumption that $A$ is level, and therefore we can conclude that $A$ has the weak Lefschetz property. 
\end{proof}

\begin{prop}\label{pr14a41}
  Any artinian Gorenstein algebra with Hilbert function $(1,4,a,4,1)$ has the weak Lefschetz property if $\operatorname{char}k\ne 2$. 
\end{prop}

\begin{proof}
Using the above notation, suppose that $A$ has Hilbert function $h=(1,4=b_1+c_1,a=b_2+c_2,4=b_3+c_3,1=b_4)$ and does not have the weak Lefschetz property. Then the quotient of $A$ by a general linear form, $A/(L)$, has $\dim_k (A/(L))_3=c_1=1$, which is the maximal possible according to Green's Theorem~\ref{1}. By Proposition~\ref{pr3} we then have that $\dim_k A_1=2$, contradicting the assumption on $A$. Thus, $c_3=0$, $a\geq 4$, and the map $\cdot L:A_2\longrightarrow A_3$ is surjective. 

Since $c_3=0$, we have that $b_3=4=b_2$, and so $(b_1,b_2,b_3,b_4)=(1,4,4,1)$. Hence $\dim_k(R/(I:L))_1=b_1=4$, that is, $\ker (\cdot L)=(I:L)/I=0$. In other words, the multiplication map $\cdot L:A_1\longrightarrow A_2$ is injective.

Assume $L\cdot A_3=0$. Then $L$ is a socle element of $A$ in degree 3, which is a contradiction since $A$ is Gorenstein. In other words, the multiplication map $\cdot L:A_3\longrightarrow A_4$ is surjective.
\end{proof}

\begin{prop} 
    Let $A$ be a level algebra with Hilbert function $(h_0,h_1,h_2,\dots,h_d)$. If 
  $$
  h_d= \binom{m+d}{d}, \qquad h_{d-1}=\binom{m+d-1}{d-1}+1
  $$
  then $A$ has the weak Lefschetz property if $\operatorname{char}k\ne 2$. 
\end{prop}

\begin{proof}
  Suppose that the multiplication by a general linear form, $\cdot L:A_{d-1}\longrightarrow A_d$,  fails to be injective.
  Then the Hilbert function of the quotient $A/(L)$ in degree $d$ is at least $$h_d-h_{d-1}+1=\binom{m+d}{d}-\binom{m+d-1}{d-1}=\binom{m-1+d}{d},$$ which is also the maximal possible according to Green's upper bound. Hence, by Proposition~\ref{pr3}, the Hilbert function in degree $d-1$ has to be $\binom{m+d-1}{d-1}$, contradicting the assumption. We conclude that $\ker (\cdot L)$ is trivial in degree $d-1$, and  since $A$ is level, it has to be trivial in all lower degrees as well.
\end{proof}

\begin{ex}
  Every level algebra with Hilbert function starting with
$$
  1,4,7,11,15, \dots
$$
has the weak Lefschetz property at least up to degree three. 
\end{ex}

The next result is an application of Proposition \ref{pr4}. It proves the non-existence of two specific Gorenstein $h$-vectors of very small socle degree whose nature was still unknown. This further extends the class of pairs (codimension, socle degree) where the inequality \cite[Theorem 4]{MNZ1} is not sharp.

\begin{prop}\label{444} Assume that $\operatorname{char}k=0$. 
\begin{itemize}
\item[$i)$] All Gorenstein $h$-vectors of the form $(1,14,a,a,14,1)$ are unimodal.
\item[$ii)$] If $(1,18,a,t,a,18,1)$ is a Gorenstein $h$-vector, then $a\geq 18$.
\end{itemize}
\end{prop}

\begin{rem}\label{4444}
  It follows from Proposition \ref{444}, $i)$ that all Gorenstein $h$-vectors of socle degree 5 and codimension $r\leq 14$ are unimodal, extending \cite[Corollary 6]{MNZ1} from $r=13$ to $r=14$.
\end{rem}

\begin{proof}[Proof of Proposition \ref{444}]
$i)$ Suppose that $h=(1,14,a,a,14,1)$ is not unimodal, i.e. that $a\leq 13$. Again, let us write $h=b+c$ with the usual notation. Applying Green's Theorem \ref{1}, we obtain $c_4\leq \left(14_{(4)}\right)_0^{-1}=4$.

Now, by Proposition \ref{pr4} with $d=m=4$, we have that if $c_4=4$, then $h$ starts with $(1,3,\dots)$, which is not the case. Hence, $c_4\leq 3$. Thus, by symmetry, $b_4=b_2\geq 14-3=11$. It follows that $c_2=a-b_2\leq 13-11=2$.

By Macaulay's Theorem \ref{mac}, this actually implies that $c_4\leq 2$. Therefore, we now get $b_4=b_2\geq 14-2=12$, which similarly gives $c_2\leq 1$.

Repeating the same reasoning once more, we obtain $c_4\leq 1$, hence $b_2\geq 13$, which forces $b_2=13=a$ and $c_2=a-13=0$. Thus, $b_4=13$, and therefore $c_4=14-13=1$, a contradiction to Macaulay's theorem. This proves part $i)$.

$ii)$ The idea is similar. We suppose that $h=(1,18,a,t,a,18,1)$ is Gorenstein with $a\leq 17$, and seek a contradiction. By writing $h=b+c$, we now consider $c_5$, which satisfies $c_5\leq \left(18_{(5)}\right)_0^{-1}=4$ by Theorem \ref{1}.

But, if $c_5=4$, by Proposition \ref{pr4} with $d=5$ and $m=4$, we would have $h$ of codimension 3 and not 18. Therefore $c_5\leq 3$.

The symmetry of $b$ gives $b_2=b_5=18-c_5\geq 15$. Thus, $c_2=a-b_2\leq 17-15=2$, which implies, by Macaulay's theorem $c_4\leq 2$. The rest of the argument now follows {\em mutatis mutandis} that of part $i)$, and the proof of the proposition is complete.\end{proof}

\section*{Acknowledgements}
We are thankful to Mark Green for personal communications concerning the assumption on the characteristic in the last part of his paper~\cite{Gr}, including a new argument which improved the theorems even more than we were hoping for.   

We would like to thank Tony Iarrobino for several valuable comments on the draft version of this paper. We are also grateful to the referee for useful comments, and to Giulio Caviglia for providing new proofs~\cite{Ca} of Green's theorems~\cite[Theorems 3 and 4]{Gr} in characteristic zero using generic initial ideals. The project started with a visit of the second author to the first author supported by a grant from the G{\"o}ran Gustafsson Foundation.

\section{Appendix}
Green has provided us with new arguments which complete the proofs of Theorem 3 and Theorem 4, and has given us permission to include a sketch of those arguments based on his ideas, following also the lines of the proofs contained in the original paper. 

\begin{proof}[Sketch of proof of Theorem 3]
The proof goes with induction on the degree $d$ and on $m$, where $W$ is a linear subspace of codimension $c=\binom{d+m}{d}$ in $H^0(\mathcal O_{\mathbb P^r}(d))$.

By induction we can assume that for two general linear hyperplanes $H$ and $H'$ we get 
$$
W_{H'}(-(H\cap H'))=I_{d-1}(-P_H\cap H'),
$$
where the linear space $P_H$, depending on $H$, has dimension $m$. In order to finish the proof, we have to conclude that $P_H$ is in fact  constant and does not depend on $H$. 

We use that $P_{H_1}\cap H = P_{H_2}$, for  a general hyperplane $H$ in the pencil of hyperplanes spanned by $H_1$ and $H_2$. For each point $p$ in $P_{H_1}$ not in $H_1$, we can see that $p$ also has to be in $H_2$ by taking $H$ to be the hyperplane spanned by $H_1\cap H_2$ and $p$.  Thus, this shows that $P_{H_1}=P_{H_2}$, unless $P_H\subseteq H$ for a general hyperplane $H$. Since 
$W(-H)=I_{d-1}(P_H) = \ell(H)R_{d-1}$,
we have that $P_H\subseteq H$ implies  $\ell(H)^2R_{d-2}\subseteq W$, where $\ell(H)$ is the linear form corresponding to the hyperplane $H$. If the characteristic of $k$ is different from $2$, the squares of the linear forms spans $R_2$, which shows that $P_H\not\subseteq H$, for a general $H$. 
\end{proof}

\begin{proof}[Sketch of proof of Theorem 4]
The proof uses induction on $d$ and starts in the case where $d=k$. Using the same technique as in the beginning of the previous theorem, we get that $W(-H)=I_{k-1}(P_H)$, where $P_H$ is a two-dimensional linear space depending on $H$. We can use the same idea as above to show that, in fact, $P_H$ does not depend on $H$ when the characteristic of the field is different from $2$. We then see that $W=I_d(P)+(F)$, for some form $F$, since $\dim_k W = \dim_k I_d(P)+1$.

Now, we consider the case $d>k$, where  by induction one can assume that the ideal is given by $W(-H)=I_{d-1}(C_H)$, where $C_H$ is a plane curve in a fixed plane, but potentially dependent on the hyperplane $H$. We need to prove that $C_H=C$.  

We have $$W_H(-(H_1\cap H)) = I_{d-1}(C_{H_1}\cap H) = W_H(-(H_2\cap H)) = I_{d-1}(C_{H_2}\cap H),$$
which proves that 
$$
C_{H_1}\cap H = C_{H_2}\cap H
$$
for $d>k$. As before, we take a general point $p$ of a component of $C_{H_1}$ outside $H_1$ and look at the hyperplane $H$ spanned by $p$ and $H_1\cap H_2$. Then we have that $p\in C_{H_1} \cap H = C_{H_2}\cap H$, meaning that $p$ also lies on $C_{H_2}$. If none of the components of $C_{H_1}$ is contained in $H_1$, we conclude that $C_{H_1}=C_{H_2}$.  In general, we have to pass to the algebraic closure of $k$ to find points on $C_H$, but this is not a problem since the conclusion of the theorem is still valid over the original field $k$.

A curve in the plane is completely determined by its intersections with the general elements of a pencil of lines through any point not on the curve, unless the curve has a linear component through the point. Thus we only need to exclude the cases where $C_H$ has a linear component in $H$, i.e., if $F_H = \ell(H)^m G_H$, for some $G_H$ not divisible by $\ell(H)$. Now, $W:\ell(H)=W(-H) = I_{d-1}(C_H)=\ell(H)^mG_HR_{d-k-1}$ means that $\ell(H)^{m+1}G_HR_{d-k-1}\subseteq W$. Hence we use the previous argument to show that $G_H=G$, for some polynomial $G$, since this does not have any component contained in $H$ for a general $H$. Thus we have that 
$$
\ell(H)^{m+1}G R_{d-k-1} \subseteq W,
$$
for a general $H$. If the characteristic is zero, or sufficiently large compared to the degree, we know that the $m+1$ powers of the linear forms span  $R_{m+1}$, and $W$ contains $G\cdot R_{d-k+m}$. However, $\dim_k W= \dim_k R_{d-k}$, which means that $m=0$. We need that the characteristic does not divide $m+1$, which is at most $k+1$. In order for this  to be true for any $m=1,2,\dots,k$, we have that the characteristic is larger than $k+1$. 
\end{proof}

\begin{rem}
  In characteristic $p$, there are examples that show that Theorem 4 cannot be extended. Let $W = (x^pG,y^pG,z^pG)_d\subseteq R=k[x,y,z]$, where $G\in R_{d-p}$. Then we have that
$$
\operatorname{codim} W = \binom{d+2}{2}-3 = \binom{d+1}{d}+\binom{d}{d-1}+\cdots+\binom{3}{2}
$$
and 
$$
\operatorname{codim} W_H = d-1.
$$
\end{rem}


\end{document}